\documentclass[a4paper,11pt]{article}
\usepackage[utf8]{inputenc}
\usepackage{a4wide}
\usepackage{amssymb}
\usepackage{amsmath}
\usepackage{amsthm}
\usepackage[inline]{enumitem}
\usepackage{latexsym}
\usepackage{mathrsfs}
\usepackage[all]{xy}
\usepackage{hyperref}
\usepackage{todonotes}

\theoremstyle{plain}
\newtheorem{theorem}{Theorem}
\newtheorem{defthm}[theorem]{Definition and Theorem}
\newtheorem{proposition}[theorem]{Proposition}
\newtheorem{lemma}[theorem]{Lemma}
\theoremstyle{definition}
\newtheorem{definition}[theorem]{Definition}
\theoremstyle{remark}
\newtheorem{remark}[theorem]{Remark}
  
\newcommand{\Q}{\mathbb Q}  
\newcommand{\R}{\mathbb R}  
\newcommand{\C}{\mathbb C}  
\newcommand{\N}{\mathbb N} 
\newcommand{\half}{\mathbb H}

\newcommand{\coleq}{\mathrel{\mathop:}=}

\newcommand{\abso}[1]{\left|#1\right|}

\newcommand{\eps}{\varepsilon}
\newcommand{\vphi}{\varphi}

\newcommand{\D}{\mathcal{D}}

\title{Laplace transformation of vector-valued distributions and applications to Cauchy-Dirichlet problems} 

\author{Michael Kunzinger\footnote{University of Vienna, Faculty of Mathematics, michael.kunzinger@univie.ac.at}, Eduard Nigsch\footnote{University of Vienna, Faculty of Mathematics, eduard.nigsch@univie.ac.at}, Norbert Ortner\footnote{University of Innsbruck, Department of Mathematics, mathematik1@uibk.ac.at}
}

\newcommand{\sL}{\mathscr{L}}
\newcommand{\sS}{\mathscr{S}}
\newcommand{\sB}{\mathscr{B}}
\newcommand{\sO}{\mathscr{O}}
\newcommand{\sD}{\mathscr{D}}
\newcommand{\sH}{\mathscr{H}}
\newcommand{\sK}{\mathscr{K}}
\newcommand{\pd}{\partial}

\date{\today}

\DeclareFontFamily{U}{mathx}{\hyphenchar\font45}
\DeclareFontShape{U}{mathx}{m}{n}{
      <5> <6> <7> <8> <9> <10>
      <10.95> <12> <14.4> <17.28> <20.74> <24.88>
      mathx10
      }{}
\DeclareSymbolFont{mathx}{U}{mathx}{m}{n}
\DeclareFontSubstitution{U}{mathx}{m}{n}
\DeclareMathAccent{\widecheck}{0}{mathx}{"71}
\DeclareMathAccent{\wideparen}{0}{mathx}{"75}

\begin{document}

\maketitle

\begin{abstract}
We present two new proofs of the exchange theorem for the Laplace transformation of 
vector-valued distributions. We then derive an explicit solution to the Dirichlet problem of the 
polyharmonic operator in a half-space. Finally, we obtain explicit solutions to Cauchy-Dirichlet 
problems of iterated wave- and Klein-Gordon-operators in half-spaces.

%Our study follows three goals:
%\begin{itemize}
%\item Presentation of two alternative proofs of the exchange theorem for the Laplace transformation of 
%vector-valued distributions.
%\item Explicit solution to the Dirichlet problem of the polyharmonic operator in a half-space.
%\item Explicit solutions to Cauchy-Dirichlet problems of iterated wave- and Klein-Gordon-operators
%in half-spaces.
%\end{itemize}

\vskip 1em

\medskip

\end{abstract}

\section{Introduction}
The exchange theorem for the Laplace transformation $\sL$ states that
\[ \sL ( S * T) = \sL S \cdot \sL T \qquad (S \in \sS'(\Gamma), T \in \sO_C'(\Gamma)) \]
with notation as explained in Section \ref{sec:new_proof} below (cf.\ \cite[Prop.~7, p.~308]{SCH3}). 
A first task of this study is to present two new proofs for the exchange theorem for vector-valued distributions $S$ and $T$. The original proof was given by L.~Schwartz in his theory of vector-valued distributions. Our first proof follows an idea indicated at the beginning of L.~Schwartz' proof, namely to apply a proposition on the convolution of two vector-valued distributions $S \in \sH(E)$, $T \in \sK(F)$, in which both the space $\sH$ and its strong dual $\sH'_b$ are assumed to be nuclear. Our second proof relies on a theorem of R.~Shiraishi on the convolution of vector-valued distributions that supposes only $\sH$ 
(and not necessarily $\sH'_b$) to be nuclear.

Our second goal is to use the Laplace transform of vector-valued distributions to deduce explicit solution formulae for the Cauchy-Dirichlet problem of the operators
\begin{align*}
 & (\Delta_n + \pd_y^2 - \pd_t^2)^m \qquad & & \textrm{(iterated wave operator)} \\
 & (\Delta_{2n+1} + \pd_y^2 - \pd_t^2 - \xi^2)^m \qquad & & \textrm{(iterated Klein-Gordon operator)}\\
\intertext{in the half-space $y >0$ (Propositions \ref{prop:3} and \ref{prop:4}) by starting with the solutions of the Dirichlet problem of the elliptic operator}
 & (\Delta_n + \pd_y^2 - p^2)^m \qquad & & \textrm{(iterated metaharmonic operator)}
\end{align*}
(Propositions \ref{prop:2} and \ref{prop:2'}).
We assume the Cauchy data $u|_{t=0}$, $\pd_t u|_{t=0}$, $\dotsc$, $\pd_t^{2m-1} u|_{t=0}$ to vanish. In the terminology of R.~Courant and D.~Hilbert such problems are called ``transient response problems'' \cite[p.~224]{CH}. Compare also \cite[p.~85]{H}. The expression ``metaharmonic'' is borrowed from \cite{garnir} and from \cite{vekua}.

In Proposition \ref{Edenhofer} we recall the distributionally formulated solution to the Dirichlet problem of the iterated Laplace (i.e., polyharmonic) operator $(\Delta_n + \pd_y^2)^m$ in the half-space $y>0$, which was presented for the first time in \cite{E}.

We note that our method could also be used, e.g., to derive explicit formulae for the solution to the Cauchy-Dirichlet problem of the operator
\[
 (\Delta_n + \pd_y^2 - \pd_t)^m \qquad \textrm{(iterated heat operator)}
\]
in the half-space $y>0$.

For $m=1$, the solution of the Cauchy-Dirichlet problem can be found by an odd extension of the sought solution in the half-space, application of the distributional differentiation formula and convolution with the fundamental solution (in classical terms, by application of a representation theorem by means of Green's function). If $m>2$ the solution by extension is not known to us.

%\editodo{notation}
Our notation is standard, mostly following \cite{SCH1,SCH2,SCH3}.

\section{New proofs of L.~Schwartz' exchange theorem for the Laplace transform of the convolution of vector-valued distributions.}\label{sec:new_proof}

Let us first recall L.~Schwartz' version \cite[Prop.~43, p.~186]{SCH2}:

\begin{theorem}\label{thm1}
 Let $\Gamma$ be a non-void open convex subset of $\R^n$. Let $E$ and $F$ be separated locally convex topological vector spaces. Then there is a hypocontinuous (with respect to bounded sets) convolution mapping
\[ \overset{*}{\otimes} \colon \sS'(\Gamma)(E) \times \sS'(\Gamma)(F) \to \sS'(\Gamma)(E \wideparen \otimes_\pi F). \]

For two Laplace-transformable distributions $S \in \sS'(\Gamma)(E)$, $T \in \sS'(\Gamma)(F)$ with values in $E$ and $F$, respectively, and their Laplace images $\sL S, \sL T$ we have
\[ \sL ( S \overset{*}{\otimes} T) = (\sL S) \overset{\cdot}{\otimes} ( \sL T). \]
\end{theorem}

We will now explain the notions appearing in this theorem. First, the mappings $\overset{*}{\otimes}$ and $\overset{\cdot}{\otimes}$ in Theorem \ref{thm1} are defined as in \cite[Proposition 3, p.~37]{SCH2} and would be denoted by $*_\pi$ and $\cdot_\pi$ there, respectively. Also, $E \wideparen \otimes_\pi F$ denotes the quasi-completion of $E\otimes_\pi F$.

\begin{definition}[{\cite[p.~58]{SCH1}, \cite[p.~303]{SCH3}}]\label{def1}
 Let $\Gamma \subseteq \R^n$ be non-void and convex. The space of \emph{Laplace transformable distributions} $\sS'(\Gamma)$ is defined as
\[ \sS'(\Gamma) = \bigcap_{\xi \in \Gamma} e^{\xi x} \sS'_x = \{ S \in \sD'\ |\ \forall \xi \in \Gamma: e^{-\xi x} S(x) \in \sS'_x \}, \]
where $\sS'$ is the space of temperate distributions on $\R^n$. $\sS'(\Gamma)$ is endowed with the projective topology with respect to the linear maps $\sS'(\Gamma) \to \sS'$, $S(x) \mapsto e^{-\xi x } S(x)$ for $\xi \in \Gamma$.
\end{definition}

As usual, we denote by $\sO_C'$ the space of rapidly decreasing distributions on $\R^n$. By defining
\[ \sO_C'(\Gamma) = \bigcap_{\xi \in \Gamma} e^{\xi x} \sO_{C,x}' \]
analogously to $\sS'(\Gamma)$ we have $\sO'_C(\Gamma) = \sS'(\Gamma)$ if $\Gamma \ne \emptyset$ is convex and open \cite[p.~303]{SCH3}.
Also, for such $\Gamma$ the space $\sO_C'(\Gamma)$ is a commutative algebra with respect to convolution, which in turn is continuous \cite[Corollaire, p.~304]{SCH3}.

Let us recall L.\ Schwartz' definition of the Laplace transformation of \emph{scalar valued} distributions and the Paley-Wiener-Schwartz theorem:
\begin{defthm}[{\cite[Prop.~22, p.~76]{SCH1}, \cite[Prop.~6, p.~306]{SCH3}}]\label{defthm2}
 Let $\emptyset\not=\Gamma \subseteq \R$ be open and convex and $T^\Gamma \coleq \Gamma + i \R^n$ the tube domain over $\Gamma$. The \emph{Laplace transformation} $\sL$ is the mapping
 \[ \sL \colon \sO_C'(\Gamma) \to \sH(T^\Gamma),\quad S \mapsto \sL S(p) = \langle 1(x), e^{-px}S(x) \rangle,\quad p \in T^\Gamma. \]
 The vector-valued scalar product $\langle\ ,\ \rangle$ is defined on $\sO_C \times \sO_C'(\sH(T^\Gamma))$ due to $e^{-px} S(x) \in \sO'_{C,x}(\sH(T^\Gamma_p))$ \cite[Cor., p.~302]{SCH3}.
 
 $\sL$ is an isomorphism if $\sH(T^\Gamma)$ is endowed with the projective topology (with respect to the compact subsets $K$ of $\Gamma$) of the inductive limits
 \[ \sH(T^K) = \{ f \colon T^K \to \C\ \textrm{holomorphic}\ |\ \exists m \in \N_0: (1+\abso{p}^2)^{-m} f(p) \in L^\infty(T^K) \}. \]
\end{defthm}

% A comma normally follows participial phrases that introduce a sentence:
% When an adverbial phrase begins a sentence, it’s often followed by a comma but it doesn’t have to be, especially if it’s short. As a rule of thumb, if the phrase is longer than about four words, use the comma
Concerning the proof of Theorem \ref{thm1}, L.~Schwartz remarks that it could be realized by applying Proposition 3 in \cite[p.~37]{SCH2} to the spaces $\sS'(\Gamma)(E)$ and $\sS'(\Gamma)(F)$. But this procedure would require the proof of the nuclearity of the space $\sS'(\Gamma)$ ``which is easy'' and of the nuclearity of its strong dual ``which is not so easy''. Thus, he proceeds differently \cite[p.~187]{SCH2}. However, we aim at performing the proof in such a manner as L.~Schwartz remarked. As a byproduct we sharpen Theorem \ref{thm1} slightly.

\begin{theorem}\label{thm3}
 Let $\emptyset \ne \Gamma \subseteq \R^n$ be an open and convex set and $E$ and $F$ separated locally convex topological vector spaces. There exists a \emph{unique}  bilinear, \emph{continuous} mapping
 \[ \overset{*}{\otimes} \colon \sO_C'(\Gamma)(E) \times \sO_C'(\Gamma)(F) \to \sO_C'(\Gamma)(E \wideparen \otimes_\pi F) \]
which extends the mapping
 \begin{align*}
  \overset{*}{\otimes} \colon (\sO_C'(\Gamma) \otimes E) \times (\sO_C'(\Gamma) \otimes F) &\to \sO_C'(\Gamma) (E \otimes_\pi F)\\
  (S\otimes e,T\otimes f) & \mapsto (S*T) e\otimes f,
 \end{align*} 
wherein $* \colon \sO_C' (\Gamma) \times \sO_C'(\Gamma) \to \sO_C'(\Gamma)$ is the continuous convolution and $\otimes \colon E \times F \to E \otimes_\pi F$ the canonical bilinear and continuous mapping.
 
 If $\sL \colon \sO_C'(\Gamma)(E) \to \sH(T^\Gamma)(E)$ is the Laplace transformation of $E$-valued distributions then we have
 \[ \sL(S \overset{*}{\otimes} T) = \sL S \overset{\cdot}{\otimes} \sL T \]
 for $S \in \sO_C'(\Gamma)(E)$, $T \in \sO_C'(\Gamma)(F)$. Summarizing, we  have the commutative diagram
 \[
  \xymatrix{
  \sO_C'(\Gamma)(E) \times \sO_C'(\Gamma)(F) \ar[r]^-{\overset{*}{\otimes}} \ar[d]^-{\sL \times \sL}_-{\rotatebox{90}{$\cong$}} & \sO_C'(\Gamma) (E \wideparen \otimes_\pi F) \ar[d]^-{\sL}_-{\rotatebox{90}{$\cong$}} \\
  \sH(T^\Gamma)(E) \times \sH(T^\Gamma)(F) \ar[r]^-{\overset{\cdot}{\otimes}} & \sH(T^\Gamma)(E \wideparen \otimes_\pi F).
  }
 \]
\end{theorem}

\begin{proof} %Theorem \ref{thm3} is an immediate consequence of 
The claim about the map $\overset{*}{\otimes}$ follows from
Proposition 3 in \cite[p.~37]{SCH2} because the space $\sO_C'(\Gamma)$ has the strict approximation property \cite[Proposition 16, p.~59]{SCH1}, is nuclear and its strong dual is nuclear (see Lemma \ref{lem} below), and because the convolution $\sO_C'(\Gamma) \times \sO_C'(\Gamma) \xrightarrow{*} \sO_C'(\Gamma)$ is continuous.  Concerning $\overset{\cdot}{\otimes}$, we note that the 
multiplication $\cdot : \sH(T^\Gamma)\times \sH(T^\Gamma) \to \sH(T^\Gamma)$ is continuous,
and since $\sH(T^\Gamma)\cong \sO_C'(\Gamma)$, we may again apply Proposition 3 in \cite[p.~37]{SCH2}. Finally, commutativity of the diagram follows from that of its scalar variant, which 
holds due to \cite[Proposition 7, p.\ 308]{SCH3}.
\end{proof}

\begin{lemma}\label{lem}
Let $\emptyset \ne \Gamma \subseteq \R^n$ be open and convex.
 \begin{enumerate}[label=(\roman*)]
  \item The space $\sO_C'(\Gamma)$ is nuclear and complete. %\todo{and therefore a semi-Montel and a semireflexive space -- eh klar, muss das auch gesagt werden?}
  %{\color{red}bornological, hence ultrabornological and a Montel, in particular, a reflexive space.}
  \item The strong dual $(\sO_C'(\Gamma))'_b$ of $\sO_C'(\Gamma)$ is nuclear.%{\color{red}, ultrabornological and complete}.
 \end{enumerate}
\end{lemma}

\begin{proof}
(i) The nuclearity of $\sO_C'(\Gamma)$ is an immediate consequence of \cite[Corollaire 2, p.~48]{G2} and the nuclearity of $\sO_C'$ (\cite[Th\'eor\`eme 16, p.~131]{G2}).
 
%  $\sO_C'(\Gamma)$ is isomorphic to the projective limit $\varprojlim \sO_C'(\Gamma_f)$ of the spaces $\sO_C'(\Gamma_f)$, $\Gamma_f$ being a finite subset of $\Gamma$.
By \cite[Proposition 5.3, p.~52]{Schae} the space $\sO_C'(\Gamma)$ is complete.
 
(ii) The projective limit $\sO_C'(\Gamma) = \sS'(\Gamma)$ is countable due to
 \[ \bigcap_{\xi \in \Gamma} e^{\xi x} \sS'_x = \bigcap_{\xi \in \Gamma \cap \Q^n} e^{\xi x} \sS'_x. \]

We only have to show the inclusion ``$\supseteq$'', the rest being elementary. For this, given $T \in \bigcap_{\xi \in \Gamma \cap \Q^n} e^{\xi x} \sS_x'$ and $\xi \in \Gamma$, choose $\xi_1, \dotsc, \xi_k\in \Gamma \cap \Q^n$ such that $\xi$ is in the convex hull of 
$\{ \xi_1, \dotsc, \xi_k \}$. By \cite[p.\ 301]{SCH3},
\[ e^{-\xi x} = \alpha(x,\xi) \sum_{j=1}^k e^{-\xi_j x} \]
with $\alpha(.,\xi) \in \sB$, so we have
\[ e^{-\xi x} T(x) = \alpha(x,\xi) \sum_{j=1}^k \underbrace{e^{-\xi_j x} T(x)}_{\in \sS'_x} \in \sS'_x. \] 

% {\color{red}
% (4) $\sO_C'(\Gamma)$ is barrelled by Lemma \ref{barrelled}.
%  
%  (5) By (1), (2) and (4), $\sO_C'(\Gamma)$ is nuclear, barrelled and complete, hence Montel 
% (\cite[Proposition 50.2, Corollary 3]{Trev}). In particular, it
% is a reflexive, complete Schwartz space. Hence, its strong dual $(\sO_C'(\Gamma))'_b$ is ultrabornological \cite[p.~43]{SCH1} and, in particular, barrelled. Consequently, $(\sO_C'(\Gamma))'_b$ carries the Mackey topology $\tau (( \sO_C'(\Gamma))', \sO_C'(\Gamma))$.
%  
%  (6) 
% By \cite[2.~Proposition, p.~38]{J} we may assume that $\sO_C'(\Gamma)$ is a reduced countable projective limit. Then taking into account 
% the nuclearity of $\sO_C$ \cite[Th\'eor\`eme 16, p.~131]{G2} the application of \cite[Proposition 4.4, p.~139]{Schae} yields that the strong dual $(\sO_C'(\Gamma))'_b$ of $\sO_C'(\Gamma)$ is a countable inductive limit of nuclear spaces and therefore nuclear by \cite[Corollaire 1, p.~48]{G2}. In virtue of the reflexivity of 
% $\sO_C$ the strong dual $(\sO_C'(\Gamma))'_b$ has the representation
%  \[ \left(\bigcap_{\xi \in \Gamma \cap \Q^n} e^{\xi x} \sO'_{C,x}\right)'_b \cong \sum_{\xi \in \Gamma \cap \Q^n} e^{-\xi x} \sO_{C,x}. \]
% }
 
It follows that $\sS'(\Gamma)$ is given by the projective limit
\[ \sS'(\Gamma) = \varprojlim_{\Gamma_f \subseteq \Gamma\textrm{ finite}} \sS'(\Gamma_f). \]
This limit is reduced because the inclusions $\sD \subseteq \sS'(\Gamma) \subseteq \sS'(\Gamma_f)$ and denseness of $\sD$ in $\sS'(\Gamma_f)$ imply that $\sS'(\Gamma)$ also is dense.
%  We show that $\sO_C'(\Gamma) = \sS'(\Gamma)$ is ultrabornological.
%  We enumerate $\Gamma \cap \Q^n = \{ \xi_1, \xi_2, \dotsc \}$ and, following  \cite{W}, define the projective spectrum $\cX = (X_n, \rho^n_m)$ with
%  \[ X_n = \bigcap_{i=1}^n e^{\xi_i x} \sS_x' \]
%  as well as, for $m \ge n$, the canonical inclusions $\rho^n_m \colon X_m \hookrightarrow X_n$. Then $\sS'(\Gamma) = \Proj \cX$ and we also have canonical inclusions $\rho^n: \sS'(\Gamma) \hookrightarrow X_n$ which have dense range because $\sD \subseteq \sS'(\Gamma)$ and $\sD$ is dense in $X_n$, which means that $\cX$ is reduced. Now $\sS'$ is a (DFS)-space (i.e., the dual of a Fr\'echet-Schwartz space).
% , which in the terminology of \cite{W} correponds to an (LS)-space.
% This class of spaces is stable under the formation of finite products
% and closed subspaces \cite[Theorem A.5.13, p.~253]{Morimoto}
% % (cf.\ \cite[p.\ 114, 2.2 (b),(c)]{Z}),
% so each $X_n$ is itself a (DFS)-space. 
{
By \cite[4.4, p.~139]{Schae} the dual $(\sS'(\Gamma))'$ endowed with the Mackey topology $\tau((\sS'(\Gamma))', \sS'(\Gamma))$ can be identified with the inductive limit of the spaces
\[ ((\sS'(\Gamma_f))', \tau((\sS'(\Gamma_f))', \sS'(\Gamma_f))). \]
Because $\sS'(\Gamma)$ is nuclear and complete it is semireflexive, hence the Mackey topology on its dual equals the strong topology \cite[Prop.~4, p.~228 and Prop.~8, p.~218]{Hor1}.
% Nun ist
% \[ \tau ( ( \sS'(\Gamma))', \sS'(\Gamma)) = \beta ( (\sS'(\Gamma))', \sS'(\Gamma)) \]
% because for a semireflexive, quasicomplete space $E$ we have $\tau(E', E) = \beta(E', E)$, i.e., every bounded subset of $E$ is contained in a (weakly) compact set, cf. Horvath, TVSD, Prop 4, p.228.
Hence, we have
\[ (\sS'(\Gamma))'_b = \varinjlim_{\Gamma_f \subseteq \Gamma \cap \Q^n} ( \sS'(\Gamma_f))'_b \]
and the nuclearity of $(\sO_C'(\Gamma))'_b = (\sS'(\Gamma))'_b$ follows by \cite[Cor.~1, p.~48]{G2} from the nuclearity of $(\sS'(\Gamma_f))'_b$. To see that the latter space is nuclear we note that $\sS'(\Gamma_f)$, as a finite projective limit of (DFS)-spaces, is itself a (DFS)-space because this class of spaces is stable under the formation of finite products and closed subspaces \cite[Theorem A.5.13, p.~253]{Morimoto}. Furthermore, $\sS'(\Gamma_f)$ is nuclear by \cite[Cor.~1, p.~48]{G2}, hence its strong dual is nuclear by \cite[Th\'eor\`eme 7, p.~40]{G2}.
% . (G2, Thm. 7, p.~40).
}
% We also see that $\Proj \cX$ is strongly reduced in the sense that
%  \[ \forall n\ \exists m: \rho^n_m ( X_m) \subseteq \overline{ \rho^n (\Proj \cX ) }. \]
%  In fact, this is clear with $m=n$ from $\sD \subseteq \sO_C'(\Gamma) \subseteq X_m$ and denseness of $\sD$ in $\sO_C'(\Gamma)$ \cite[p.~304]{SCH3}, implying $X_m = \overline{ \sO_C'(\Gamma) }$.
% {\color{red}Now we can apply \cite[Theorem 3.5]{W2}
% % \cite[Corollary 3.3.10]{W},
% which states that for $\Proj \cX$, bornologicity is equivalent to barrelledness; as the latter was established in (4), $\sO_C'(\Gamma)$ is ultrabornological, taking into account its completeness and \cite[Ch.~3, \S 15, Exercise 9.\ (b)]{Hor1}.
%  
%  (8) Due to (5), $(\sO_C'(\Gamma))'_b$ is ultrabornological by \cite[p.~43]{SCH1}, and complete by the bornologicity of $\sO_C'(\Gamma)$.
% }
\end{proof}

Further properties of the space $\sO_C'(\Gamma)$ will be published in an upcoming paper.

\begin{remark}A third proof of Theorem \ref{thm3} can be given by means of \cite[Theorem 2, p.~196]{Shir}, see also \cite[Theorem 5, p.~18]{BO}. Compared to 
\cite[Prop.\ 3, p.\ 37]{SCH3} it has the advantage that the nuclearity of $\sO_C'(\Gamma) = \sS'(\Gamma)$ is sufficient (Lemma \ref{lem} (i)),  while nuclearity of
its strong dual $\sO_C'(\Gamma)'_b$ need not be established.
Instead, one has to show that $\sO_C'(\Gamma)$ is $\dot B$-normal (which is straightforward) and that $\sO_C'(\Gamma) \otimes E$ is strictly dense in $\sO_C'(\Gamma)(E)$, which in turn is implied by the strict approximation property of $\sS'(\Gamma)$ (\cite[Proposition 16, p.\ 59]{SCH3}). In fact, by using \cite[Prop.\ 1, p.\ 19]{BO} we can even dispense with showing $\dot B$-normality
of $\sO_C'(\Gamma)$.
\end{remark}

\section{Poisson kernels for Dirichlet problems}
Out next aim is to reformulate a known result on the Poisson kernels of the Dirichlet problems of polyharmonic operators in half-spaces and to apply the partial Fourier transformation in order to deduce the Poisson kernels of the Dirichlet problems of the iterated metaharmonic operators in half-spaces. Then the theory of vector-valued distributions is applied in order to continue the results analytically. This method goes back to H.\ G.\ Garnir \cite{garnir}.

We follow the terminology
of \cite[p.\ 635]{ADN} and \cite[p.\ 140]{Shi}: The \emph{Poisson kernel} of the $j$-th Dirichlet problem for the
operator 
\[
(\Delta_n + \partial_y^2 - \xi^2)^m, \quad \Delta_n = \partial_1^2+\dots +\partial_n^2,\quad m,n\in \N,\ \xi\in \R
\]
in the half-space $\half=\{(x,y)\in \R^{n+1}: x\in \R^n,\ y>0\}$, $j=0,\dots,m-1$, is the distribution $E_j\in \D'(\half)$ for which  
\begin{align*}
(\Delta_n + \partial_y^2 - \xi^2)^m E_j &= 0 \qquad E_j \in \  \mathcal{O}_M(\half) = \{\varphi \in \mathcal{E}(\half) \ |\ \forall \alpha \in \N_0^n\ \exists k \in \N_0: \\
& \qquad \qquad \qquad (1+\abso{x}^2)^{-k/2} \partial^\alpha \varphi \in \mathcal{C}_0(\half) \}  \\
\partial_y^k E_j|_{y=0} &= \delta(x)\delta_{jk}, \quad k=0,\dots,m-1 \ \text{ in } \D'(\R^n).
\end{align*}
Here, $\mathcal{C}_0(\half) = \{\psi\in \mathcal{C}(\half) : \lim_{|(x,y)|\to \infty} \psi(x,y) =0 \}$.
The existence of the restrictions $\pd_y^k E_j|_{y=0}$ will follow from the explicit form of $E_j$ in Proposition \ref{Edenhofer} below (see also \cite[Theorem 4.4.8, p.~115]{lars1}).

For a more general notion of Poisson kernel we refer to \cite[Section 4.5, p.~137]{Tanabe}
% \miketodo{Ich habe jetzt die Verweise auf \cite[Theorem 4.4.8, p.~115]{lars1} und \cite[Example 2, p.~12]{malgrange}
% gestrichen, weil diese, wie Eduard richtig feststellt nur auf $|_{y=y_0}$ für $y_0>0$ anwendbar wären. OK?}
%\editodo{eigentlich: $|_{y=y_0}$ for $y_0>0$} is a consequence of \cite[Theorem 4.4.8, p.~115]{lars1} or of \cite[Example 2, p.~12]{malgrange}, together with \cite{E} (see below).

To begin with, we use \cite[Satz 3]{E} to derive the following result:
\begin{proposition}\label{Edenhofer} The Poisson kernel of the $j$-th Dirichlet problem for the polyharmonic (i.e., iterated
Laplace) operator $(\Delta_n+\partial_y^2)^m$ is given by
\[
E_j = \frac{2}{\omega_{n+1}} \frac{y^m}{j!(m-1-j)!} (-\partial_y)^{m-j-1} \left(\frac{1}{(|x|^2+y^2)^{\frac{n+1}{2}}}\right),
\]
where $\omega_{n+1} = \frac{2\pi^{\frac{n+1}{2}}}{\Gamma(\frac{n+1}{2})}$ denotes the surface measure of the unit sphere
in $\R^{n+1}$.
\end{proposition}
\begin{proof}
Let $\vphi\in \D(\R^n_x)$, $\check{\vphi}(x)=\vphi(-x)$, and denote by $*$ the convolution with respect to the
$x$-variables. Then it follows from \cite[Satz 3]{E} that $E_j*\check{\vphi}\in \mathcal{E}_{xy}(\half)$ 
is the unique solution to  
\begin{align*}
(\Delta_n+\partial_y^2)^m(E_j*\check{\vphi}) &= 0 \\
\lim_{y\searrow 0} \partial_y^k(E_j*\check{\vphi})(x) &= \check{\vphi}(x)\delta_{jk}, \quad k=0,1,\dots,m-1.
\end{align*}
Consequently, $(\Delta_n+\partial_y^2)^m E_j=0$ and $\lim_{y\searrow 0} \partial_y^k\langle E_j,\vphi\rangle = \vphi(0)\delta_{jk}$,
$k=0,1,\dots, m-1$, i.e., $\lim_{y\searrow 0} \partial_y^k  E_j  = \delta(x)\delta_{jk}$, $k=0,1,\dots, m-1$.
\end{proof}
\begin{remark}
We point out the following particular cases of Proposition \ref{Edenhofer}:
\begin{itemize}
\item [(a)] For $m=1$, $j=0$ we obtain the well-known Poisson kernel of the Dirichlet-problem for $\Delta_n+\partial_y^2$
in the half-space $y>0$ to be
\[
\frac{\Gamma\big(\frac{n+1}{2}\big)}{\pi^{\frac{n+1}{2}}} \frac{y}{(|x|^2 + y^2)^{\frac{n+1}{2}}},
\]
cf., e.g., \cite[(1.2), p.\ 163]{Shi} or \cite[p.\ 37, Th.\ 14]{Evans}.
\item[(b)] The choice $m=2$, $j=0$, $j=1$ gives the Poisson kernels of the Dirichlet problem for the \emph{biharmonic operator}
$(\Delta_n+\partial_y^2)^2$ in the half-space $y>0$: 
\[
E_0 = \frac{2\Gamma\big(\frac{n+3}{2}\big)}{\pi^{\frac{n+1}{2}}} \frac{y^3}{(|x|^2 + y^2)^{\frac{n+3}{2}}}, \quad
E_1 = \frac{\Gamma\big(\frac{n+1}{2}\big)}{\pi^{\frac{n+1}{2}}} \frac{y^2}{(|x|^2 + y^2)^{\frac{n+1}{2}}}.
\]
In \cite{E}, J.~Edenhofer cites \cite{BF} for this result.

The solution $u$ to 
\begin{align*}
(\partial_x^2+\partial_y^2)^2u &= 0 \quad \text{ in } y>0\\
u|_{y=0} &= g_0, \ \partial_y u|_{y=0} = g_1 \quad (n=1,\ m=2,\ j=0,1),
\end{align*}
in the form
\[
u(x,y) = \frac{2y^3}{\pi} \int_\R g_0(x-\xi) \frac{d\xi}{(\xi^2+y^2)^2} + \frac{y^2}{\pi} \int_\R g_1(x-\xi) \frac{d\xi}{\xi^2+y^2}
\]
can be found in \cite[(2.14), p.\ 262]{P} (where a sign should be corrected and where the formula is attributed to
L.\ F.\ Richardson in \cite{Richardson}). For a more recent, direct treatment
of the Poisson kernel $E_0$ for the biharmonic operator in the half-plane we refer to \cite[p.\ 781]{A}. 
\item[(c)] If $n=2$, the Poisson kernel $E_{m-1}$ of the Dirichlet problem for $(\partial_x^2 + \partial_y^2)^m$ in $y>0$
(i.e., with the boundary conditions $\partial_y^k E_{m-1}|_{y=0} = 0$, $k=0,1,\dots,m-2$, $\partial_y^{m-1}E_{m-1}|_{y=0} = \delta(x)$)
is given by the formula
\[
E_{m-1} = \frac{1}{\pi(m-1)!} \frac{y^m}{x^2+y^2},
\]
see Example 5 in \cite[p.\ 275]{Sh}. 
\end{itemize}
\end{remark}
Next, let us derive the Poisson kernel of the $j$-th Dirichlet problem, $j=0,1,\dots,m-1$, of the operator $(\Delta_n+\partial_y^2-\xi^2)^m$
in $\half$ by Fourier transformation.
\begin{proposition}\label{prop:2}
Let $m, n\in \N$. The \emph{Poisson kernel} of the $j$-th Dirichlet problem, $j=0,1,\dots,m-1$, for the \emph{iterated meta-harmonic operator}
$(\Delta_n+\partial_y^2-\xi^2)^m$ in $\half$, $\xi\not=0$, is given by
\begin{equation}\label{2.1}
E_j = \frac{y^m |\xi|^{\frac{n+1}{2}}}{2^{\frac{n-1}{2}}\pi^{\frac{n+1}{2}}j!(m-1-j)!}(-\partial_y)^{m-j-1}
\left(\frac{K_{\frac{n+1}{2}}(|\xi|\sqrt{|x|^2+y^2})}{(|x|^2+y^2)^{\frac{n+1}{4}}}\right)
\end{equation}
or
\begin{equation}\label{2.2}
E_j = \frac{-y^m |\xi|^{\frac{n-1}{2}}}{2^{\frac{n-1}{2}}\pi^{\frac{n+1}{2}}j!(m-1-j)!}
(-\partial_y)^{m-j-1}\left(\frac{1}{y}\partial_y\right)
\left(\frac{K_{\frac{n-1}{2}}(|\xi|\sqrt{|x|^2+y^2})}{(|x|^2+y^2)^{\frac{n-1}{4}}}\right)
\end{equation}
Here, $K_\lambda$ is the modified Bessel function of the second kind of order $\lambda$.
\end{proposition}
\begin{proof}
Setting
\begin{align*}
F_j &:=\frac{2}{\omega_{n+2}}\frac{y^m}{j!(m-1-j)!}(-\partial_y)^{m-j-1} \frac{1}{(|x|^2+y^2+s^2)^{\frac{n}{2}+1}}\\
&=-\frac{2}{n\omega_{n+2}}\frac{y^m}{j!(m-1-j)!}(-\partial_y)^{m-j-1}\left(\frac{1}{y}\partial_y\right) 
\frac{1}{(|x|^2+y^2+s^2)^{\frac{n}{2}}},
\end{align*}
we obtain by means of Proposition \ref{Edenhofer}:
\begin{align*}
(\Delta_n+\partial_y^2+\partial_s^2)^m F_j &=0 \\
\partial_y^k F_j|_{y=0} &= \delta(x,s)\delta_{jk}, \quad k=0,1,\dots,m-1.
\end{align*}
A partial Fourier transformation with respect to $s$ yields for $E_j=\int_\R e^{-i\xi s}F_j\,ds$:
\begin{align*}
(\Delta_n+\partial_y^2-\xi^2)^m E_j &=0 \\
\partial_y^k E_j|_{y=0} &= \delta(x)\delta_{jk}, \quad k=0,1,\dots,m-1.
\end{align*}
By \cite[8.432,5]{GR} we obtain
\begin{align*}
E_j &= \frac{\Gamma(\frac{n}{2}+1)}{\pi^{\frac{n}{2}+1}}\frac{y^m}{j!(m-1-j)!}(-\partial_y)^{m-j-1}
\left(\frac{2\sqrt{\pi}|\xi|^{\frac{n+1}{2}}}{2^{\frac{n+1}{2}}\Gamma(\frac{n}{2}+1)}
\frac{K_{\frac{n+1}{2}}(|\xi|\sqrt{|x|^2+y^2})}{(|x|^2+y^2)^{\frac{n+1}{4}}}\right)\\ 
&= \frac{y^m |\xi|^{\frac{n+1}{2}}}{2^{\frac{n-1}{2}}\pi^{\frac{n+1}{2}}j!(m-1-j)!}(-\partial_y)^{m-j-1}
\left(\frac{K_{\frac{n+1}{2}}(|\xi|\sqrt{|x|^2+y^2})}{(|x|^2+y^2)^{\frac{n+1}{4}}}\right),
\end{align*}
establishing \eqref{2.1}. The second claim follows from the functional relation
$(\frac{1}{z}\partial_z)(z^{-\lambda}K_\lambda(z))=-z^{-\lambda-1}K_{\lambda+1}(z)$ (cf.\ \cite[8.486,15]{GR}):
\begin{align*}
E_j &= \frac{-\Gamma(\frac{n}{2}+1)y^m}{n\pi^{\frac{n}{2}+1}j!(m-1-j)!}(-\partial_y)^{m-j-1}
\left(\frac{1}{y}\partial_y\right)
\left(\frac{2\sqrt{\pi}|\xi|^{\frac{n-1}{2}}}{2^{\frac{n-1}{2}}\Gamma(\frac{n}{2})}
\frac{K_{\frac{n-1}{2}}(|\xi|\sqrt{|x|^2+y^2})}{(|x|^2+y^2)^{\frac{n-1}{4}}}\right)\\ 
&= \frac{-y^m |\xi|^{\frac{n-1}{2}}}{2^{\frac{n-1}{2}}\pi^{\frac{n+1}{2}}j!(m-1-j)!}(-\partial_y)^{m-j-1}\left(\frac{1}{y}\partial_y\right)
\left(\frac{K_{\frac{n-1}{2}}(|\xi|\sqrt{|x|^2+y^2})}{(|x|^2+y^2)^{\frac{n-1}{4}}}\right).
\end{align*}
\end{proof}
\begin{remark}
Let us mention a particular case of Proposition \ref{prop:2}: Setting $m=1$, $j=0$, the Poisson kernel $E_0$ of the
\emph{metaharmonic operator} $\Delta_n  + \partial_y^2 - \xi^2$ in $\half$ is given by 
\[
E_0 = \frac{y |\xi|^{\frac{n+1}{2}}}{2^{\frac{n-1}{2}}\pi^{\frac{n+1}{2}}}
\frac{K_{\frac{n+1}{2}}\left(|\xi|\sqrt{|x|^2+y^2}\right)}{(|x|^2+y^2)^{\frac{n+1}{4}}},
\]
see \cite[Rem.\ 2, p.\ 321]{BD}.
\end{remark}
Proposition \ref{prop:2} remains valid if we substitute $p\in T^\Gamma = \R_+ + i\R$ ($\Gamma=\R_+=(0,\infty)$, $T^\Gamma$ the 
right half-plane) for $\xi\in \R \setminus \{0\}$. We obtain by analytic continuation:
\begin{proposition}\label{prop:2'}
The Poisson kernel $E_j$ of the $j$-th Dirchlet problem, $j=0,1,\dots,m-1$, of the \emph{iterated meta-harmonic operator}
$(\Delta_n+\partial_y^2-p^2)^m$, $p\in T^\Gamma$, in the half-space $\half$ is given by
\[
E_j = \frac{-y^m p^{\frac{n-1}{2}}}{2^{\frac{n-1}{2}}\pi^{\frac{n+1}{2}}j!(m-1-j)!}
(-\partial_y)^{m-j-1}\left(\frac{1}{y}\partial_y\right)
\left(\frac{K_{\frac{n-1}{2}}(p\sqrt{|x|^2+y^2})}{(|x|^2+y^2)^{\frac{n-1}{4}}}\right)
\]
(For even $n$ the square root in $p^{\frac{n-1}{2}}$ is defined as usual.) Furthermore, 
$E_j\in \mathcal{H}(T^\Gamma_p)(\mathcal{D}'(\half_{xy}))$.
\end{proposition}
\begin{proof}
The integral representation
\[
\frac{K_{\frac{n-1}{2}}(p\sqrt{|x|^2+y^2})}{(|x|^2+y^2)^{\frac{n-1}{4}}} = \frac{1}{2}\int_0^\infty t^{-\frac{n+1}{2}}
e^{-\frac{p}{2}\left(t+\frac{|x|^2+y^2}{t}\right)}\,dt
\]
in \cite[8.432,7]{GR} can be interpreted as a vector-valued scalar product 
\[
\frac{1}{2} \left\langle 1(t), t^{-\frac{n+1}{2}} e^{-\frac{p}{2}t}\cdot e^{-\frac{p}{2t}\left(|x|^2+y^2\right)} \right\rangle
\]
on $L^\infty(\R_{+,t})\times L^1(\R_{+,t})(\mathcal{H}(T^\Gamma_p)(\mathcal{D}'(\half_{xy})))$. Here, $e^{-\frac{p}{2}t}\in 
\mathcal{H}(T^\Gamma_p)(L^1(\R_{+,t}))$, and 
\begin{align*}
S(p,t,x,y):=
e^{-\frac{p}{2t}\left(|x|^2+y^2\right)}t^{-\frac{n+1}{2}}
&\in \mathcal{H}(T^\Gamma_p)(\mathcal{C}_0(\R_{+,t})(\mathcal{D}'(\half_{xy}))) \\
& = \mathcal{H}(T^\Gamma_p) \hat\otimes_\eps \mathcal{C}_0(\R_{+,t}) \hat\otimes_\eps \mathcal{D}'(\half_{xy}).
\end{align*}
To prove this, we first show the following two auxiliary results:

\begin{lemma}\label{Mai2018}Any complete, nuclear normal space of distributions has the $\eps$-property.
\end{lemma}
\begin{proof}
By the K\=omura Theorem \cite[21.7.1, p.~500]{J}, any such space $F$ is isomorphic to a closed subspace of $s^J$ for some index set $J$. Since the $\eps$-property is preserved under taking products and subspaces, this implies that $F$ has the $\eps$-property.
\end{proof}

\begin{lemma}
$\mathcal{H}(T^\Gamma_p) \hat\otimes \mathcal{D}'(\half_{xy})$ has the $\eps$-property.
\end{lemma}
\begin{proof}
We have $\mathcal{H}(T^\Gamma) \subseteq \mathcal{E}(\Gamma\times\R^n)$, with the topology induced by $\mathcal{E}$.
Both $\mathcal{E}$ and $\mathcal{D}$ are nuclear normal spaces of distributions, so $\mathcal{E}\hat\otimes \mathcal{D}'$
is nuclear, normal, and complete and has the $\eps$-property by Lemma \ref{Mai2018}. As the $\eps$-property is inherited by topological subspaces, the claim follows.
\end{proof}
To establish $S(p,t,x,y)\in (\mathcal{H}(T^\Gamma_p) \hat\otimes \mathcal{D}'(\half_{xy}))(\mathcal{C}_0(\R_{+,t}))$ it therefore
suffices to show that for $\mu\in \mathcal{M}^1(\R_{+,t})$ we have
\[
\langle S(p,t,x,y),\mu(t) \rangle \in \mathcal{H}(T^\Gamma_p) \hat\otimes \mathcal{D}'(\half_{xy}).
\]
Noting that $\mathcal{H}(T^\Gamma_p)$ has the $\eps$-property, to see this it suffices in turn to show that
\begin{equation}\label{star}
\langle \vphi(x,y), \langle S(p,t,x,y),\mu(t)\rangle\rangle \in \mathcal{H}(T^\Gamma_p).
\end{equation}
for each $\vphi\in \mathcal{D}(\half_{xy})$. By Fubini's theorem (\cite[p.\ 131, Corollaire]{SCH1}) this
is equivalent to
\begin{equation*}
\langle \langle S(p,t,x,y),\vphi(x,y)\rangle, \mu(t)\rangle \in \mathcal{H}(T^\Gamma_p).
\end{equation*}
In fact,
\begin{align*}
\langle t^{-\frac{n+1}{2}}e^{-\frac{p}{2t}\left(|x|^2+y^2\right)},\vphi(x,y)\rangle &= \int_{\half_{xy}\times \R_{+,t}}
e^{-\frac{p}{2t}\left(|x|^2+y^2\right)}t^{-\frac{n+1}{2}}\vphi(x,y)\,dxdy \\ 
&= 
\int_{\half_{xy}\times \R_{+,t}}
e^{-\frac{p}{2}\left(|\xi|^2+\eta^2\right)}\vphi(\sqrt{t}\xi,\sqrt{t}\eta)\,d\xi d\eta,
\end{align*}
so that
\[
\langle \mu(t), \langle S(p,t,x,y),\vphi(x,y)\rangle\rangle = \int_{\half_{xy}\times \R_{+,t}}
e^{-\frac{p}{2}\left(|\xi|^2+\eta^2\right)}\langle \mu(t),\vphi(\sqrt{t}\xi,\sqrt{t}\eta)\rangle\,d\xi d\eta.
\]
Since the map $\half \to \C$, $(\xi,\eta)\mapsto \langle\mu(t),\vphi(\sqrt{t}\xi,\sqrt{t}\eta)\rangle$
is bounded by $\|\mu\|_1 \|\vphi\|_\infty$, \eqref{star} follows by dominated convergence.
Note that this argument in fact also shows that  
$S(p,t,x,y)\in \mathcal{H}(T^\Gamma_p)(\mathcal{B}\mathcal{C}_b(\R_{+,t})(L^1(\half_{xy})))$.
Here, the subscript $b$ refers to the Buck topology, so $\mathcal{B}\mathcal{C}_b(\R_{+})' = \mathcal{M}(\R_+)$ 
(cf. \cite[p.\ 6]{OW13}).

Due to the continuity of the bilinear multiplication $\mathcal{H}(T^\Gamma)\times \mathcal{H}(T^\Gamma) \stackrel{\cdot}{\rightarrow} 
\mathcal{H}(T^\Gamma)$ and the continuity of the vector-valued multiplication $\mathcal{C}_0(\R_{+,t})(\mathcal{D}'(\half_{xy}))\times
L^1(\R_{+,t}) \stackrel{\cdot}{\rightarrow} L^1(\R_{+,t})(\mathcal{D}'(\half_{xy}))$ we conclude by means of \cite[Proposition 3, p.\ 37]{SCH2} 
that
\[
e^{-\frac{pt}{2} -\frac{p}{2t}\left(|x|^2+y^2\right)} t^{-\frac{n+1}{2}} \in 
 \mathcal{H}(T^\Gamma_p)(L^1(\R_{+,t})(\mathcal{D}'(\half_{xy})).
\]
Indeed, the assumptions `$\mathcal{H}(T^\Gamma)$ nuclear' and `$(\mathcal{H}(T^\Gamma))_b'$ nuclear' are fulfilled
because of $\mathcal{H}(T^\Gamma)\cong \mathcal{O}_C'(\Gamma)$ (Definition and Theorem \ref{defthm2}) and Lemma \ref{lem}.

In virtue of
\[
\mathcal{H}(T^\Gamma_p)(L^1(\R_{+,t})(\mathcal{D}'(\half_{xy}))) = L^1(\R_{+,t})(\mathcal{H}(T^\Gamma_p)(\mathcal{D}'(\half_{xy}))),
\]
the final step consists in applying \cite[Theorem 7.1, p.\ 31]{mixed} to the vector-valued scalar product
$\langle \, , \,\rangle : L^\infty \times L^1(E)\to E$, with $E=\mathcal{H}(T^\Gamma_p)(\mathcal{D}'(\half_{xy}))$.
\end{proof}

\section{Transient response Dirichlet problems}
In this section we study the transient response Dirichlet problem for the iterated wave operator $(\Delta_n+\partial_y^2-\partial_t^2)^m$
and the iterated Klein-Gordon operator $(\Delta_n+\partial_y^2-\partial_t^2-\xi^2)^m$ in the half-space $y>0$, more precisely
in $\half_{yt} = \{(x,y,t)\in \R^{n+2}: y>0,\ t>0 \}$.

We look for an explicit expression for the solution $E_j$ to the $j$-th, $j=0,1,\dots, m-1$, (mixed) Cauchy-Dirichlet problem
\begin{align*}
&(\Delta_n + \partial_y^2 - \partial_t^2 - \xi^2)^m E_j = 0 \qquad \text{ in } \mathcal{D}'(\half_{yt}) \\
&E_j|_{t=0} = \partial_tE_j|_{t=0} = \dots =\partial^{2m-1}_tE_j|_{t=0} = 0\ \text{ in } \mathcal{D}'(\half_1)\\
&\partial_y^k E_j|_{y=0} = \delta(x,t)\delta_{jk}, \quad k=0,\dots,m-1, \ \text{ in } \mathcal{D}'(\half_2),
\end{align*}
where $\half_1=\{(x,y)\in \R^{n+1} : y>0 \}$, $\half_2=\{(x,t)\in \R^{n+1} : t>0 \}$.

For the general theory of the mixed problem for constant coefficient, linear partial differential operators
see \cite[12.9, p.\ 162--179]{Ho2} and \cite[p.\ 57--118]{Sa}. We call $E_j$ \emph{Poisson kernel} of the 
Cauchy-Dirichlet problem for the iterated wave operator and the iterated Klein-Gordon-operator if $\xi=0$
or $\xi\not=0$, respectively (\cite[p.\ 94]{Sa}).
\begin{proposition}\label{prop:3}
The \emph{Poisson kernel} $E_j$, $j=0,1,\dots,m-1$, of the \emph{Cauchy-Dirichlet problem for the iterated 
wave operator} $(\Delta_n+\partial_y^2-\partial_t^2)^m$ in the half space $\half_{yt}$ is given by 
\[
E_j = \frac{-y^m }{2^{n-1}\pi^{\frac{n}{2}}\Gamma(\frac{n}{2}) j!(m-1-j)!}
(-\partial_y)^{m-j-1}\left(\frac{1}{y}\partial_y\right)
\left(\frac{\partial_t^{n-1}(t^2-|x|^2-y^2)_{+}^{\frac{n}{2}-1}}{(|x|^2+y^2)^{\frac{n-1}{2}}}{Y(t)} \right),
\]
where $x_+^\lambda:=x^\lambda Y(x)$. Furthermore, $E_j\in \mathcal{S}'(\R_{+,t})(\mathcal{D}'(\half_{1,xy}))$.
\end{proposition}
\begin{proof} Recall from Definition \ref{def1} that
\[
\mathcal{S}'(\R_{+,t}) := \bigcap_{\tau\in (0,\infty)} e^{\tau t}\mathcal{S}'_t,
\] 
so by Definition and Theorem \ref{defthm2}, the Laplace transform
\[
\sL: \mathcal{S}'(\R_{+,t})(\mathcal{D}'(\half_{xy})) \to \mathcal{H}(T^\Gamma_p)(\mathcal{D}'(\half_{xy}))
\]
is an isomorphism.

Thus the inverse Laplace transform $E_j:=\sL^{-1}F_j$ of the Poisson kernel $F_j$ of the $j$-th Dirichlet
problem in the half-space of the iterated metaharmonic operator in Proposition \ref{prop:2'} yields
\begin{align*}
&(\Delta_n + \partial_y^2 - \partial_t^2)^m E_j = 0  \\
&E_j|_{t=0} = \partial_tE_j|_{t=0} = \dots =\partial^{2m-1}_tE_j|_{t=0} = 0\\
&\partial_y^k E_j|_{y=0} = \delta(x,t)\delta_{jk}, \quad k=0,\dots,m-1, 
\end{align*}
and 
\[
E_j = \frac{-y^m }{2^{\frac{n-1}{2}}\pi^{\frac{n+1}{2}}j!(m-1-j)!}
(-\partial_y)^{m-j-1}\left(\frac{1}{y}\partial_y\right)
\left(\frac{\sL^{-1}\left(p^{\frac{n-1}{2}}K_{\frac{n-1}{2}}(p\sqrt{|x|^2+y^2})\right)}{(|x|^2+y^2)^{\frac{n-1}{4}}}\right).
\]
By \cite[(5.19)]{CL} we have for any $S\in \mathcal{H}(T^\Gamma)$: $\sL^{-1}(p^{n-1}S) = \partial_t^{n-1}\sL^{-1}S$.
Hence, we obtain by means of the transform pair
% \editodo{wie lässt sich die punkt\-weise Formel distributionell interpretieren?} in \cite[p.\ 121]{CL}:
\begin{align*}
\sL^{-1}\left(p^{\frac{n-1}{2}}K_{\frac{n-1}{2}}(p\sqrt{|x|^2+y^2})\right) &= \partial_t^{n-1}
\sL^{-1}\left(\frac{K_{\frac{n-1}{2}}(p\sqrt{|x|^2+y^2}))}{p^{\frac{n-1}{2}}}\right)\\
&= \partial_t^{n-1} \left(\frac{\sqrt{\pi}}{\Gamma(\frac{n}{2})} 
\frac{(t^2-|x|^2-y^2)_{+}^{\frac{n}{2}-1}}{(2\sqrt{|x|^2+y^2})^{\frac{n-1}{2}}} {Y(t)} \right).
\end{align*}
In fact, \cite[(5.19)]{CL} is the inverse relation of
\[
\left\langle 1(t), e^{-pt} \frac{(t^2-|x|^2-y^2)}{(|x|^2+y^2)^{\frac{n-1}{4}}} Y(t-\sqrt{|x|^2-y^2})\right\rangle
\frac{\sqrt{\pi}}{2^{\frac{n-1}{2}}\Gamma(n/2)} = \frac{K_{\frac{n-1}{2}}(p\sqrt{|x|^2+y^2})}{p^\frac{n-1}{2}}
\]
(which can be seen using \cite[8.432,3]{GR} or \cite[\S 6.15. (4), p.~172]{Watson}).
That this identity is in fact valid in $\mathcal{H}(T^\Gamma_p)\hat\otimes \mathcal{D}'(\half_{xy})$ can be concluded
similarly to the proof of Proposition \ref{prop:2'}.
This yields the formula stated in the Proposition. The fact that $E_j$ belongs to
$\mathcal{S}'(\R_{+,t})(\mathcal{D}'(\half_{1,xy}))$ now follows by inspection.

It remains to show that $\left.\partial_t^k E_j\right|_{t=0} = 0$ for $0\le k \le 2m-1$. For this we first note
that by \cite[Th.\ 12.9.12, p.~176]{Ho2} we have $E_j \in \mathcal{C}^\infty ([0,\infty),\mathcal{D}'(\half_{xy}))$, so
\[
\partial_t^k E_j(t) = \mathrm{const} \cdot y^m (-\partial_y)^{m-j-1} \left(\frac{1}{y}\partial_y\right) \partial_t^{k+n-1}
\frac{Y(t)(t^2-|x|^2-y^2)_+^{\frac{n}{2} - 1}}{(|x|^2+y^2)^{\frac{n-1}{2}}} \in \mathcal{D}'(\half_{xy})
\]
and for $\vphi \in \mathcal{D}(\half_{xy})$ we obtain
\begin{align*}
\langle \vphi, \partial_t^k E_j(t)\rangle &= \mathrm{const}\cdot  \partial_t^{k+n-1} Y(t) 
 \left\langle \partial_y \frac{1}{y} \partial_y^{m-j-1} (y^m\vphi),   
\frac{(t^2-|x|^2-y^2)_+^{\frac{n}{2} - 1}}{(|x|^2+y^2)^{\frac{n-1}{2}}}\right\rangle \\
&= \mathrm{const}\cdot \partial_t^{k+n-1} \left(Y(t) \int_{|x|^2+y^2\le t^2} \phi(x,y) 
\frac{(t^2-|x|^2-y^2)^{\frac{n}{2} - 1}}{(|x|^2+y^2)^{\frac{n-1}{2}}}\,dxdy\right),
\end{align*}
where $\phi(x,y):=\partial_y \frac{1}{y} \partial_y^{m-j-1} (y^m\vphi)$. Applying the homothety 
$x=t\xi$, $y=t\eta$, $t>0$ shows that the latter equals
\[
\mathrm{const}\cdot \partial_t^{k+n-1} \left(t_+^n \int_{|\xi|^2+\eta^2\le 1} \phi(t\xi,t\eta) 
\frac{(1-|\xi|^2-\eta^2)^{\frac{n}{2} - 1}}{(|\xi|^2+\eta^2)^{\frac{n-1}{2}}}\,dxdy\right),
\]
so
\begin{align*}
\lim_{t\searrow 0} \langle \vphi, \partial_t^k E_j(t)\rangle =  \mathrm{const}\cdot \int_{|\xi|^2+\eta^2\le 1}
\left[\lim_{t\searrow 0} \partial_t^{k+n-1} t_+^n \phi(t\xi,t\eta) \right] \cdot 
\frac{(1-|\xi|^2-\eta^2)^{\frac{n}{2} - 1}}{(|\xi|^2+\eta^2)^{\frac{n-1}{2}}}\, d\xi d\eta.
\end{align*}
As $\phi$ vanishes at $t=0$, together with all its derivatives, we indeed arrive at 
$\left.\partial_t^k E_j\right|_{t=0} = 0$ for all $k$.

\end{proof}
\begin{remark}
We single out two important special cases:
\begin{itemize}
\item[(a)] $n=1$, $m=1$, $j=0$:

The Poisson kernel of the mixed problem
\begin{align*}
&(\partial_x^2 + \partial_y^2 - \partial_t^2) E_0 = 0,\quad x\in \R,\ y>0,\ t>0, \\
&E_0|_{t=0} = \partial_tE_0|_{t=0} =  0\\
&E_0|_{y=0} = \delta(x,t) 
\end{align*}
in the half-space $y>0$ is given by 
\[
E_0=-\frac{1}{\pi} \partial_y \frac{Y(t)}{(t^2-x^2-y^2)_{+}^\frac{1}{2}} = -\frac{Y(t)}{\pi}\partial_y
\left((t^2-x^2-y^2)_+^{-\frac{1}{2}} \right).
\]
In \cite[Ex.\ 405, p.\ 189]{LSU} the solution $U$ to the mixed problem with the temporally constant boundary
value $U|_{y=0} = \delta(x)$ is presented. It emerges from $E_0$ by convolution with $\delta(x)\otimes Y(t)$, i.e.,
\begin{align*}
U =-\frac{Y(t)}{\pi}\partial_y \left((t^2-x^2-y^2)_+^{-\frac{1}{2}} \right) * (\delta(x)\otimes Y(t))
= \frac{yt_+}{\pi (x^2+y^2)(t^2-x^2-y^2)_+^{1/2}}.
\end{align*}
Note that our derivation differs essentially from that proposed in \cite{LSU}, where Fourier- and Laplace
transformation are suggested to be applied with respect to different variables.
\item[(b)] $n=2$, $m=1$, $j=0$:

The Poisson kernel of the mixed Cauchy-Dirichlet problem 
\begin{align*}
&(\Delta_2 + \partial_y^2 - \partial_t^2) E_0 = 0,\quad x\in \R^2,\ y>0,\ t>0, \\
&E_0|_{t=0} = \partial_tE_0|_{t=0} =  0\\
&E_0|_{y=0} = \delta(x,t) 
\end{align*}
in the half-space $y>0$ is given by 
\[
E_0 = \frac{-Y(t)}{2\pi}\partial_y\left( \frac{1}{\sqrt{|x|^2+y^2}} \partial_t (Y(t^2-|x|^2-y^2)) \right)
= \frac{-1}{2\pi t} \partial_y(\delta(t-\sqrt{|x|^2+y^2})).
\]
Note that $E_0$ is the negative derivative in the direction normal to the boundary of the Green-function of the mixed
problem of $\Delta_2 + \partial_y^2 - \partial_t^2$ in the half-space $y>0$ (compare \cite[p.\ 92]{Ga}).
We obtain the solution $U$ of the mixed problem with a temporally constant boundary value, $U|_{y=0}=\delta(x)$,
by convolution of $E_0$ with $\delta(x) \otimes Y(t)$:
\[
U = \frac{-Y(t)}{2\pi}\partial_y\left( \frac{Y(t^2-|x|^2-y^2)}{\sqrt{|x|^2+y^2}}  \right)
= \frac{-1}{2\pi} \partial_y\left( \frac{Y\left(t- \sqrt{|x|^2+y^2)}\right)}{\sqrt{|x|^2+y^2}}  \right),
\]
which coincides with \cite[p.\ 7]{BO}.
\end{itemize}
\end{remark}
The main idea in deriving the Poisson kernel of the Cauchy-Dirichlet problem $(\Delta_n+\partial_y^2-\partial_t^2)^m$ 
in the half-space $y>0$ (cf.\ the proof of Proposition \ref{prop:3}) is the application of the inverse Laplace
transformation to the Poisson kernel of the Dirichlet problem of $(\Delta_n+\partial_y^2-p^2)^m$ in $y>0$.
The Poisson kernel of the Cauchy-Dirichlet problem of the iterated Klein-Gordon operator 
$(\Delta_{2n+1}+\partial_y^2-\partial_t^2-\xi^2)^m$ ($\xi>0$) in the half-space $y>0$ can be
derived by the same method, using the Poisson kernel of  $(\Delta_{2n+1}+\partial_y^2-p^2-\xi^2)^m$
in $y>0$.
\begin{proposition}\label{prop:4}
The $j$-th \emph{Poisson kernel} $E_j$, $j=0,1,\dots,m-1$, of the Cauchy-Dirichlet problem of the \emph{iterated
Klein-Gordon operator} $(\Delta_{2n+1}+\partial_y^2-\partial_t^2-\xi^2)^m$ ($m\in \N$, $n\in \N_0$, $\xi>0$)
in the half-space $\half$ is given by
\begin{align*}
E_j = \frac{-y^m }{(2\pi)^{n+\frac{1}{2}}j!(m-1-j)!} &
(-\partial_y)^{m-j-1}\left(\frac{1}{y}\partial_y\right)
\left(\frac{1}{(|x|^2+y^2)^{\frac{n}{2}}} \sum_{l=0}^{n} \binom{n}{l} \xi^{n-2l+1} \right.\\
& \left. \partial_t^{2l} (Y(t)(t^2-|x|^2)-y^2)_+^{\frac{n}{2}-\frac{1}{4}} J_{n-\frac{1}{2}}(\xi\sqrt{t^2-|x|^2-y^2})\right),
\end{align*}
where $J_\lambda$ denotes the Bessel function of the first kind of order $\lambda$.
\end{proposition}
\begin{proof}
Similar to the proof of Proposition \ref{prop:3} we have by means of Proposition \ref{prop:2'}
\begin{align*}
E_j = \frac{-y^m }{2^n\pi^{n+1}j!(m-1-j)!}
&(-\partial_y)^{m-j-1}\left(\frac{1}{y}\partial_y\right)\\
&\left(\sL^{-1} 
\left(\frac{(p^2+\xi^2)^{\frac{n}{2}}K_{n}(\sqrt{(p^2+\xi^2)(|x|^2+y^2}))}{(|x|^2+y^2)^{\frac{n}{2}}}\right)\right)\\
= \frac{-y^m }{2^n\pi^{n+1}j!(m-1-j)!}
&(-\partial_y)^{m-j-1}\left(\frac{1}{y}\partial_y\right)\\
&\left(\sum_{l=0}^{n} \binom{n}{l} \xi^{2n-2l}\partial_t^{2l} 
\sL^{-1} 
\left(\frac{K_{n}(\sqrt{(p^2+\xi^2)(|x|^2+y^2}))}{\left((p^2+\xi^2)(|x|^2+y^2)\right)^{\frac{n}{2}}}\right)\right).
\end{align*}
By the formula
\[
\sL^{-1} 
\left(\frac{K_{n}\left(\beta\sqrt{(p^2+\xi^2)}\right)}{(p^2+\xi^2)^{\frac{n}{2}}}\right)
=\sqrt{\frac{\pi}{2}}Y(t) \xi^{-n+\frac{1}{2}}\beta^{-n} (t^2-\beta^2)_+^{\frac{n}{2}-\frac{1}{4}} 
J_{n-\frac{1}{2}}\left(\xi \sqrt{t^2-\beta^2}\right)
\]
in \cite[p.\ 125]{CL} we obtain
\begin{align*}
E_j=\frac{-y^m}{(2\pi)^{n+\frac{1}{2}}j!(m-1-j)!} & \sum_{l=0}^{n} \binom{n}{l} \xi^{n-2l+\frac{1}{2}} 
(-\partial_y)^{m-j-1}\left(\frac{1}{y}\partial_y\right)\partial_t^{2l} \\
& \left( \frac{Y(t)}{(|x|^2 + y^2)^{\frac{n}{2}}} (t^2-|x|^2 -y^2)_+^{\frac{n}{2}-\frac{1}{4}} 
J_{n-\frac{1}{2}}\left(\xi \sqrt{t^2-|x|^2-y^2}\right) \right),
\end{align*}
establishing our claim.
\end{proof}
\begin{remark}
\begin{itemize}
\item[(a)] In the special case $n=0$, $m=1$, $j=0$ we obtain
\[
E_0 = \frac{Y(t)}{\sqrt{2\pi}} \xi^{1/2}\partial_y \left( \frac{J_{-1/2}\left( \xi\sqrt{t^2-x^2-y^2} \right)}{(t^2-x^2-y^2)_+^{1/4}} \right)
= \frac{Y(t)}{\pi} \partial_y \left( \frac{\cos\left( \xi\sqrt{t^2-x^2-y^2} \right)}{(t^2-x^2-y^2)_+^{1/2}} \right)
\]
as the Poisson kernel of the Cauchy-Dirichlet problem of $\partial_x^2 + \partial_y^2-\partial_t^2-\xi^2$ ($\xi>0$) in $y>0$.
The solution $U$ to %the Cauchy-Dirichlet problem of $\partial_x^2 + \partial_y^2-\partial_t^2-\xi^2$ 
this problem in $y>0$
with the temporally constant boundary value $U|_{y=0} = \delta(x)$ emerges from $E_0$ by convolution with 
$\delta(x) \otimes Y(t)$, i.e.,  $U=E_0 * (Y(t)\otimes \delta(x,y))$. 

Note that for the Cauchy-Dirichlet problem of the related operator $\partial_x^2 + \partial_y^2-\partial_t^2-b \partial_t$
in $y>0$ with a temporally constant boundary value, the solution is given explicitly in \cite[Ex.\ 406, p.\ 189]{LSU} in terms
of elementary functions.
\item[(b)] The Poisson kernel of the Cauchy-Dirichlet problem of the iterated Klein-Gordon operator
$(\Delta_{2n} + \partial_y^2-\partial_t^2-\xi^2)^m$ in odd space dimensions can be deduced from that in 
Proposition \ref{prop:4} by J.\ Hadamard's method of descent, i.e., by integration with respect to the variable $x_{2n+1}$.
\end{itemize}
\end{remark}

{\bf Acknowledgments.} We thank Christian Bargetz for several helpful discussions, and in particular for
the idea to apply \cite[4.4, p.~139]{Schae} in the proof of Lemma \ref{lem}. We furthermore thank Andreas Debrouwere for pointing out an error in an earlier version of this paper. This work was supported by projects P26859 and P30233 of the Austrian Science Fund FWF.

\end{document}